\documentclass[12pt]{article}
\usepackage{amsmath}
\usepackage{amssymb}
\usepackage{amsthm}
\usepackage{verbatim}
\usepackage{mathrsfs}

\newcommand{\0}{{\mathcal O}}

\newcommand{\R}[0]{\mathbb R}

\newcommand{\Ds}[0]{\mathcal D}

\newtheorem{Th}{Theorem}[section]
\newtheorem{Lemma}{Lemma}[section]
\newtheorem{Prop}[Lemma]{Proposition}

\begin{document}

\title{On the regularity of the solution map of the Euler-Poisson system}
\author{H. Inci}

\maketitle

\begin{abstract}
In this paper we consider the Euler-Poisson system (describing a plasma made of ions with a negligible ion temperature) on the Sobolev spaces $H^s(\R^3)$, $s > 5/2$. Using a geometric approach we show that for any time $T > 0$ the corresponding solution map, $(\rho_0,u_0) \mapsto (\rho(T),u(T))$, is nowhere locally uniformly continuous. On the other hand it turns out that the trajectories of the ions are analytic curves in $\R^3$.
\end{abstract}

\section{Introduction}\label{section_introduction}

The initial value problem for the Euler-Poisson system in $\R^3$ is given by
\begin{align}
\label{ion}
\begin{split}
\rho_t + \operatorname{div}(\rho u) &= 0\\
u_t + (u \cdot \nabla) u &= -\nabla \phi \\
e^\phi - \Delta \phi &= \rho\\
u(0)=u_0,\;&\rho(0)=\rho_0
\end{split}
\end{align}
where $\phi:\R \times \R^3 \to \R$ is the electric potential, $\rho:\R \times \R^3 \to \R$ the ion density and $u=(u_1,u_2,u_3):\R \times \R^3 \to \R^3$ the ion velocity. System \eqref{ion} is well-posed in the Sobolev spaces $H^s, s>5/2$ -- see \cite{lwp}. More precisely, given $(\rho_0,u_0) \in (1+H^{s-1}(\R^3)) \times H^s(\R^3;\R^3)$ with $\rho_0(x) > 0$ for all $x \in \R^3$, there is $T > 0$ such that we have a unique solution
\[
(\rho,u) \in C\left([0,T];(1+H^{s-1}(\R^3)) \times H^s(\R^3;\R^3)\right)
\]
to \eqref{ion}. Moreover the solution map 
\[
(\rho_0,u_0) \mapsto (\rho(T),u(T))
\]
is continuous. To state our main result we denote the time $T$ solution map by $\Phi_T$. Its domain of definition around the equilibrium point $(1,0) \in (1+H^{s-1}(\R^3)) \times H^s(\R^3;\R^3)$ by $U_T \subseteq (1+H^{s-1}(\R^3)) \times H^s(\R^3;\R^3)$. We then have

\begin{Th}\label{thm_nonuniform}
Let $s > 5/2$ and $T > 0$. Then the solution map
\[
 \Phi_T:U_T \to (1+H^{s-1}(\R^3)) \times H^s(\R^3;\R^3), \quad (\rho_0,u_0) \mapsto (\rho(T),u(T))
\]
is nowhere locally uniformly continuous.
\end{Th}

We also show the following phenomenon which occurs also in other hydrodynamic models (see \cite{constantin})

\begin{Th}\label{analytic_trajectory}
The trajectories of the ions are analytic curves in $\R^3$.
\end{Th}

To establish our results we give a geometric formulation of \eqref{ion} following a similar procedure as in \cite{lagrangian, b_family, sqg}. The strategy is to express \eqref{ion} in Lagrangian coordinates, i.e. in terms of the flow map
\[
 \varphi_t = u \circ \varphi, \quad \varphi(0)=\operatorname{id}
\]
where $\operatorname{id}:\R^3 \to \R^3$ is the identity map. This approach was popularized by the works \cite{arnold,ebin_marsden} for the incompressible Euler equations. In a second step we use the fact that a modified vorticity is ''transported'' by the flow to produce perturbations which are irregular in some sense. 

\section{Geometric Formulation}\label{section_geometric}

To motivate the geometric formulation let us assume that $(\rho,u)$ is a solution to \eqref{ion}. Consider the flow map of $u$, i.e.
\[
 \varphi_t = u \circ \varphi,\quad \varphi(0)=\operatorname{id}
\]
We take the $t$ derivative in
\[
 \frac{d}{dt} (\det(d\varphi) \rho \circ \varphi)=(\operatorname{div}u) \circ \varphi \det(d\varphi) \rho \circ \varphi + \det(d\varphi) (\rho_t + (u \cdot \nabla) \rho) \circ \varphi=0 
\] 
where $d\varphi$ denotes the jacobian of $\varphi$. Reexpressing gives
\[
 \rho = \left(\frac{\rho_0}{\det(d\varphi)}\right) \circ \varphi^{-1}
\]
Denote by $\phi=F^{-1}(\rho)$ the solution for $\phi$ in \eqref{ion}. We can thus write
\[
u_t + (u \cdot \nabla) u = - \nabla F^{-1}(\left(\frac{\rho_0}{\det(d\varphi)}\right) \circ \varphi^{-1}) 
\]
Using $\varphi_{tt} = (u_t + (u \cdot \nabla)u) \circ \varphi$
\[
 \varphi_{tt} = -\left(\nabla F^{-1}(\left(\frac{\rho_0}{\det(d\varphi)}\right) \circ \varphi^{-1})\right) \circ \varphi 
\]
Our goal is to prove that the right hand side above is an analytic expression in $\varphi$. But first we have to introduce the appropriate functional space for $\varphi$. Consider for $s > 5/2$
\[
\Ds^s(\R^3) = \{ \varphi:\R^3 \to \R^3 \;|\; \varphi - \operatorname{id} \in H^s(\R^3;\R^3) \mbox{ and } \det(d_x\varphi) > 0 \;\forall x \in \R^3 \}
\] 
By the Sobolev imbedding theorem this spaces consists of $C^1$ diffeomorphisms. As a subset of $\operatorname{id}+H^s(\R^3;\R^3)$ it gets a differential structure and is a connected topological group under composition -- see \cite{composition}. We know that every $u \in C([0,T];H^s(\R^3;\R^3))$ generates a unique flow $\varphi \in C^1([0,T];\Ds^s(\R^3))$ with $\varphi(0)=\operatorname{id}$ -- see \cite{lagrangian}. To handle the elliptic equation in \eqref{ion} we use the lemma from \cite{lwp}
\begin{Lemma}\label{lemma_elliptic}
Let $\rho \in 1+H^{s-1}(\R^3)$ with $\rho > 0$ in $\R^3$. Then there is a unique $\phi \in H^{s+1}(\R^3)$ with
\[
 e^\phi-\Delta \phi = \rho
\]
\end{Lemma}
\noindent
Denote this $\phi$ by $F^{-1}(\rho)$. We also introduce the following open subset of $H^{s-1}(\R^3)$
\[
 U_{\bar \rho} = \{ \bar \rho \in H^{s-1}(\R^3) \;|\; \bar \rho(x) > -1 \quad \forall x \in \R^3 \} 
\]
In particular our $\rho$ in \eqref{ion} lives in $1+U_{\bar \rho}$.
\begin{Lemma}\label{lemma_local_diff}
Let $\bar \rho_0=\rho_0-1 \in U_{\bar \rho}$. Then there is an open set $W \subseteq H^{s-2}(\R^3)$ with $\bar \rho_0 \in W$ and an open set $V \subseteq H^s(\R^3)$ such that
\[
 V \subseteq H^s(\R^3) \to W \subseteq H^{s-2}(\R^3),\quad \phi \mapsto e^\phi -1 - \Delta \phi
\]
is an analytic diffeomorphism.
\end{Lemma}

\begin{proof}
From the Banach algebra property of $H^s$ we see that
\[
 H^s(\R^3) \to H^s(\R^3),\quad \varphi \mapsto e^\phi-1 = \sum_{k=1}^\infty \frac{1}{k!} \phi^k
\]
is analytic. Therefore
\[
 \Gamma:H^s(\R^3) \to H^{s-2}(\R^3),\quad \phi \mapsto e^\phi-1-\Delta \phi 
\]
is analytic. Let 
$\phi_0=F^{-1}(\rho_0)$. The differential of $\Gamma$ at $\phi_0$ is given by
\[
 d_{\phi_0}\Gamma: H^s(\R^3) \to H^{s-2}(\R^3),\quad h \mapsto e^{\phi_0} \cdot h - \Delta h
\]
which by linear elliptic theory is seen to be an isomorphism. By the Inverse Function Theorem we get open neighborhoods of $\phi_0$ in $H^s$ resp. $\bar \rho_0$ in $H^{s-2}$ on which $\Gamma$ is an analytic diffeomorphism.
\end{proof}

For $\varphi \in \Ds^s(\R^3)$ we denote by $R_\varphi$ the linear map $f \mapsto f \circ \varphi$. This map is continuous on $H^{s'}$ for $0 \leq s' \leq s$. Note also that $R_{\varphi^{-1}}=R_\varphi^{-1}$. For $\Gamma$ introduced above, i.e.
\[
 \Gamma:H^s(\R^3) \to H^{s-2}(\R^3),\quad \phi \mapsto e^\phi-1-\Delta \phi
\]
one sees that it is injective -- see \cite{lwp}. Therefore we can use $\Gamma^{-1}$ in the following. Using the notation of Lemma \ref{lemma_local_diff} we have

\begin{Lemma}\label{lemma_inverse}
Let $\bar \rho_0 = \rho_0-1 \in U_{\bar \rho}$ and $\varphi_0 \in \Ds^s(\R^3)$. Then there is an open neighborhood $W \times V \subseteq  \Ds^s(\R^3) \times H^{s-2}(\R^3)$ of $(\varphi_0,\bar \rho_0)$ such that
\[
 W \times V \to H^s(\R^3),\quad (\varphi,\bar \rho) \mapsto R_\varphi \Gamma^{-1}(\bar \rho \circ \varphi^{-1})
\]
is analytic.
\end{Lemma}

Note that by Lemma \ref{lemma_local_diff} the expression $R_\varphi \Gamma^{-1}(\bar \rho \circ \varphi^{-1})$ is defined for $W$ and $V$ small enough.

\begin{proof}
Consider the map
\begin{equation}\label{pi}
 \Pi:(\varphi,\phi) \mapsto (\varphi,R_\varphi \Gamma(\phi \circ \varphi^{-1}))
\end{equation}
Note that
\[
 \Pi(\varphi,\phi)=(\varphi,e^\phi-1-R_\varphi \Delta (\phi \circ \varphi^{-1}))
\]
We have
\[
 R_\varphi \Delta (\phi \circ \varphi^{-1}) = \sum_{k=1}^3 R_\varphi \partial_k R^{-1}_\varphi R_\varphi \partial_k R_\varphi^{-1} \phi
\]
By the chain rule
\[
 R_\varphi \nabla (\phi \circ \varphi^{-1})=[d\varphi^\top]^{-1} \nabla \phi
\]
which shows that for $1 \leq s' \leq s$ and $k=1,2,3$
\[
 \Ds^s(\R^3) \times H^{s'}(\R^3) \to H^{s'-1}(\R^3),\quad (\varphi,\phi) \mapsto R_\varphi \partial_k R_\varphi^{-1} \phi
\]
is analytic -- see also \cite{lagrangian}. Thus $\Pi$ is analytic. Let $\phi_0=F^{-1}(\rho_0)$. The differential of $\Pi$ at $(\varphi_0,\phi_0)$ is of the form
\[
 d_{(\varphi_0,\phi_0)}\Pi (g,h) = \left(\begin{array}{cc} g & 0 \\ \ast & R_{\varphi_0} \left(e^{\phi_0 \circ \varphi_0^{-1}} \cdot (h \circ \varphi_0^{-1})-\Delta (h\circ \varphi_0^{-1})\right) \end{array}\right)
\]
showing that it is an isomorphism. By the Inverse Function Theorem we conclude that $\Pi$ is an analytic diffeomorphism in a neighborhood of $(\varphi_0,\phi_0)$ in $\Ds^s(\R^3) \times H^s(\R^3)$. As
\[
R_\varphi \Gamma^{-1}(\bar \rho \circ \varphi^{-1})
\]
is the second component of $\Pi^{-1}$ the claim follows.
\end{proof}

Note that $\det(d\varphi) \in 1 + H^{s-1}$. The map
\[
 \Ds^s(\R^3) \times (1+U_{\bar \rho}) \to H^{s-1}(\R^3),\quad (\varphi,\rho_0) \mapsto \frac{\rho_0}{\det(d\varphi)} - 1 
\]
is analytic. Thus we get from Lemma \ref{lemma_inverse} that
\[
 \Ds^s(\R^3) \times (1+U_{\bar \rho}) \to H^s(\R^3),\quad (\varphi,\rho_0) \mapsto \phi \circ \varphi = R_\varphi \Gamma^{-1}(\left(\frac{\rho_0}{\det(d\varphi)}-1\right) \circ \varphi^{-1})
\]
is analytic. But we need more regularity for $\nabla \phi$.

\begin{Lemma}\label{lemma_analytic_linear}
The map
\begin{align*}
 \Ds^s(\R^3) \times H^s(\R^3) \times H^{s-2}(\R^3) &\to H^s(\R^3)\\
 (\varphi,\phi,\bar \rho) &\mapsto R_\varphi (e^{\phi \circ \varphi^{-1}}-\Delta)^{-1}(\bar \rho \circ \varphi^{-1})   
\end{align*}
is analytic.
\end{Lemma}

\begin{proof}
We proceed as in Lemma \ref{lemma_local_diff} look at the inverse expression. That is consider 
\begin{align}
\label{omega}
\begin{split}
\Theta:\Ds^s(\R^3) \times H^s(\R^3) \times H^s(\R^3) &\to \Ds^s(\R^3) \times H^s(\R^3) \times H^{s-2}(\R^3)\\
(\varphi,\phi,\xi) &\mapsto (\varphi,\phi,R_\varphi (e^{\phi \circ \varphi^{-1}}-\Delta)(\xi \circ \varphi^{-1}))
\end{split}
\end{align}
which is analytic. Its differential is of the form
\[
 d_{(\varphi,\phi,\xi)} \Theta (g,h,f) = \left( \begin{array}{ccc} g & 0 & 0\\0 & h &0\\ \ast & \ast & R_\varphi (e^{\phi \circ \varphi^{-1}}-\Delta)(f \circ \varphi^{-1}) \end{array}\right)
\]
By the Inverse Function Theorem we get the conclusion as
\[
R_\varphi (e^{\phi \circ \varphi^{-1}}-\Delta)^{-1}(\bar \rho \circ \varphi^{-1})
\]
is the third component of $\Theta^{-1}$.
\end{proof}
Combining Lemma \ref{lemma_local_diff}, \ref{lemma_inverse} and \ref{lemma_analytic_linear} we have

\begin{Prop}\label{prop_analytic}
The map
\begin{align*}
\Ds^s(\R^3) \times (1+U_{\bar \rho}) &\to H^s(\R^3;\R^3) \\
(\varphi,\rho_0) &\mapsto R_\varphi \left(\nabla F^{-1}(\left(\frac{\rho_0}{\det(d\varphi)}\right) \circ \varphi^{-1})\right)
\end{align*}
is analytic.
\end{Prop}

\begin{proof}
Applying $\nabla$ to $e^\phi-\Delta \phi=\rho$ gives
\[
 e^\phi \cdot \nabla \phi - \Delta \nabla \phi = \nabla \rho
\]
or
\[
 \nabla \phi = (e^\phi-\Delta)^{-1} \nabla \rho
\]
and
\[
 R_\varphi \nabla \phi =R_\varphi \left( (e^{(\phi \circ \varphi) \circ \varphi^{-1}}-\Delta)^{-1} \nabla \rho\right)
\]
Substituting $\rho$ gives
\[
 \nabla \rho = \left( [d\varphi^\top]^{-1} \nabla \left(\frac{\rho_0}{\det(d\varphi)}\right) \right) \circ \varphi^{-1}
\]
As 
\[
 \Ds^s(\R^3) \times (1+U_{\bar \rho}) \to H^s(\R^3),\quad (\varphi,\rho_0) \mapsto \phi \circ \varphi
\]
and 
\[
 \Ds^s(\R^3) \times (1+U_{\bar \rho}) \to H^{s-2}(\R^3;\R^3),\quad (\varphi,\rho_0) \mapsto [d\varphi^\top]^{-1} \nabla \left(\frac{\rho_0}{\det(d\varphi)}\right)
\]
are analytic we get by Lemma \ref{lemma_analytic_linear} that
\[
 \Ds^s(\R^3) \times (1+U_{\bar \rho}) \to H^s(\R^3;\R^3),\quad (\varphi,\rho_0) \mapsto \nabla \phi \circ \varphi
\]
is analytic which concludes the proof
\end{proof}

Consider the differential equation for the variables $(\varphi,v) \in \Ds^s(\R^3) \times H^s(\R^3;\R^3)$
\begin{equation}\label{ode}
 \frac{d}{dt} \left( \begin{array}{c} \varphi \\ v \end{array}\right) = 
\left(\begin{array}{c} v \\ -R_\varphi \left(\nabla F^{-1}(\left(\frac{\rho_0}{\det(d\varphi)}\right) \circ \varphi^{-1})\right) \end{array}\right)
\end{equation}
By Proposition \ref{prop_analytic} this is an analytic ODE depending analytically on the parameter $\rho_0$. Consider the Cauchy problem with initial conditions $\varphi(0)=\operatorname{id}$ and $v(0)=u_0$. The solution $\varphi$ provides via 
\[
 u(t):=\varphi_t(t) \circ \varphi(t)^{-1} \quad \mbox{and} \quad
 \rho(t)=\left(\frac{\rho_0}{\det(d\varphi(t))}\right) \circ \varphi(t)^{-1}
\]
a solution to \eqref{ion} in $(\rho,u) \in C([0,T];(1+U_{\bar \rho}) \times H^s(\R^3;\R^3))$. Thus by the existence and uniqueness result for ODEs and the composition properties of $\Ds^s(\R^3)$ we get the local wellposedness of \eqref{ion} -- see also \cite{lagrangian}. Furthermore as the trajectories of the ions are described by the analytic curves
\[
 t \mapsto \varphi(t,x)
\]
we get Theorem \ref{analytic_trajectory}.

\section{Nonuniform dependence}\label{section_nonuniform}

We denote by $\omega$ the vorticity of $u$, i.e
\[
 \omega=\operatorname{curl} u
\]
Taking $\operatorname{curl}$ in the second equation in \eqref{ion} gives
\begin{equation}\label{vorticity_equation}
 \omega_t + (u \cdot \nabla) \omega + \operatorname{div} u \cdot \omega - (\omega \cdot \nabla)u=0
\end{equation}
Now we take the $t$ derivative in
\begin{align*}
\frac{d}{dt}\left(\det(d\varphi) [d\varphi]^{-1} \omega \circ \varphi\right)=&
\det(d\varphi) \operatorname{div}u \circ \varphi [d\varphi]^{-1} \omega \circ \varphi \\
&- \det(d\varphi) [d\varphi]^{-1} [d\varphi_t] [d\varphi]^{-1} \omega \circ \varphi\\
&+ \det(d\varphi) [d\varphi]^{-1} (\omega_t + (u \cdot \nabla)\omega) \circ \varphi 
\end{align*}
Using $d\varphi_t=du \circ \varphi \cdot d\varphi$ and \eqref{vorticity_equation} gives
\[
 \frac{d}{dt}\left(\det(d\varphi) [d\varphi]^{-1} \omega \circ \varphi\right)=0
\]
or
\begin{equation}\label{transport}
 \omega(t) = \left(\frac{1}{\det(d\varphi(t))} [d\varphi(t)] \omega_0\right) \circ \varphi(t)^{-1}
\end{equation}
where $\omega_0=\operatorname{curl} u_0$. With the help of \eqref{transport} we can express $\omega(t)$ as some sort of ''pullback'' of $\omega_0$. This will be useful for our purpose. For the time $T > 0$ we denote as above by $U_T \subseteq (1+U_{\bar \rho}) \times H^s(\R^3;\R^3)$ the domain of definition of the time $T$ solution map $\Phi_T$. With $\Psi_T$ we denote the solution map in Lagrangian coordinates, i.e.
\[
 \Psi_T:U_T \to \Ds^s(\R^3),\quad (\rho_0,u_0) \mapsto \varphi(T)
\] 
where $\varphi(T)$ denotes the time $T$ value of the $\varphi$ component of the solution in \eqref{ode} with initial conditions $\varphi(0)=\operatorname{id}$ and $v(0)=u_0$. Note that $\Psi_T$ is analytic. Later we will use the following technical lemma
\begin{Lemma}\label{lemma_dense}
There is a dense subset $S \subseteq U_T$ with the property that for each $(\rho_\bullet,u_\bullet) \in S$ we have that $u_\bullet$ is compactly supported and there is $h=(h_\rho,h_u) \in H^{s-1}(\R^3) \times H^s(\R^3;\R^3)$ and $x^\ast \in \R^3$ such that $\operatorname{dist}(x^\ast,\operatorname{supp}u_\bullet) > 2$ (i.e. the distance of the point $x^\ast$ to the support of $u_\bullet$ is greater than 2) with
\[
 \left(d_{(\rho_\bullet,u_\bullet)}\Psi_T (h)\right)(x^\ast) \neq 0
\] 
\end{Lemma}

\begin{proof}
Our strategy is to get an equation for $d_{(1,0)} \Psi$. Consider for small $|\varepsilon|$, $\bar \rho \in H^{s-1}(\R^3)$ and $\bar u \in H^s(\R^3;\R^3)$ the ODE \eqref{ode}
\[
 \frac{d}{dt} \left( \begin{array}{c} \varphi^{(\varepsilon)} \\ v^{(\varepsilon)} \end{array} \right) = \left(\begin{array}{c} v^{(\varepsilon)} \\ G(\varphi^{(\varepsilon)},1+\varepsilon \bar \rho) \end{array}\right),\quad \varphi^{(\varepsilon)}(0) = \mbox{id}, v^{(\varepsilon)}(0)=\varepsilon \bar u
\]
where $G$ is the corresponding expression from \eqref{ode}. In particular we have
\[
 \varphi^{(0)}(t)=\mbox{id},\quad v^{(0)}(t)=0 \quad \forall t \geq 0
\]
We consider the variation
\[
 \partial \varphi(t) = \left. \frac{d}{d\varepsilon} \right|_{\varepsilon=0} \varphi^{(\varepsilon)}(t),\quad 
 \partial v(t) = \left. \frac{d}{d\varepsilon} \right|_{\varepsilon=0} v^{(\varepsilon)}(t)
\]
Calculating the variation in the ODE we get
\begin{equation}\label{ode_var}
 \frac{d}{dt} \left( \begin{array}{c} \partial \varphi \\ \partial v \end{array} \right) = \left(\begin{array}{c} \partial v \\ d_{(\operatorname{id},1)} G (\partial \varphi,\bar \rho) \end{array}\right),\quad \partial \varphi(0) = 0, \partial v(0)=\bar u
\end{equation}
We calculate $d_{(\operatorname{id},1)}G$. Note that with the maps from \eqref{pi} and \eqref{omega}
\[
 G(\varphi,1+\bar \rho)=(\Theta^{-1})_3\left(\varphi,(\Pi^{-1})_2(\varphi,\frac{1+\bar \rho}{\det(d\varphi)}-1),[d\varphi^\top]^{-1} \nabla \frac{1+\bar \rho}{\det(d\varphi)}\right)
\]
Using the rules for differentiating determinants we have for $k=1,2,3$
\[
 \partial_k \frac{1}{\det(d\varphi)} = -\frac{1}{(\det(d\varphi))^2} \det(d\varphi) \operatorname{tr}([d\varphi]^{-1} \partial_k d\varphi) = -\frac{1}{\det(d\varphi)} \operatorname{tr}([d\varphi]^{-1} \partial_k d\varphi)
\]
Thus the derivative of 
\[
 \varphi \mapsto [d\varphi^\top]^{-1} \nabla \frac{1}{\det(d\varphi)}= -\frac{1}{\det(d\varphi)} [d\varphi^\top]^{-1} \left(\begin{array}{c} \operatorname{tr}([d\varphi]^{-1} \partial_1 d\varphi)\\ \operatorname{tr}([d\varphi]^{-1} \partial_2 d\varphi)\\ \operatorname{tr}([d\varphi]^{-1} \partial_3 d\varphi)\end{array}\right) 
\]
in direction of $w \in H^s(\R^3;\R^3)$ at $\varphi=\operatorname{id}$ is given by
\[
 \operatorname{tr}(dw)\left(\begin{array}{c} 1 \\ 1 \\ 1 \end{array}\right)+dw^\top \left(\begin{array}{c} 1 \\ 1 \\ 1 \end{array}\right) + \left(\begin{array}{c} \partial_1 w_1 \\ \partial_2 w_2 \\ \partial_3 w_3 \end{array}\right)-\left(\begin{array}{c} \partial_1 (\partial_1 w_1+\partial_2 w_2+\partial_3 w_3) \\ \partial_2 (\partial_1 w_1+\partial_2 w_2+\partial_3 w_3) \\ \partial_3 (\partial_1 w_1+\partial_2 w_2+\partial_3 w_3) \end{array}\right)
\]
which we denote by $B(w)$. We see that $B:H^s(\R^3;\R^3) \to H^{s-2}(\R^3;\R^3)$ is a continuous linear map. The partial derivative of $G$ with respect to $\varphi$ at $(\operatorname{id},1)$ is then given by
\begin{align*}
 \left. \partial_\varphi \right|_{(\operatorname{id},1)} G (w) &= \left[(d_{(\operatorname{id},0)}\Theta)^{-1}\right]_{31}(w)+\left[(d_{(\operatorname{id},0)}\Theta)^{-1}\right]_{32}(\ast)+\left[(d_{(\operatorname{id},0)}\Theta)^{-1}\right]_{33}(Bw) \\
 &=0+0+(1-\Delta)^{-1} B w=(1-\Delta)^{-1} B w
\end{align*}
for $w \in H^s(\R^3;\R^3)$. We introduce 
\[
 A:H^s(\R^3;\R^3) \to H^s(\R^3;\R^3),\quad w \mapsto (1-\Delta)^{-1} B w.
\]
It is a Fourier multiplier operator with a continuous $L^\infty$ multiplier
\[
 m_A(\xi_1,\xi_2,\xi_3)=\frac{1}{(1+|\xi|^2)} \left(\begin{array}{ccc} -3i\xi_1+\xi_1^2 & -i\xi_1-i\xi_2+\xi_1 \xi_2 & -i\xi_1-i\xi_3 + \xi_1 \xi_3\\
 -i\xi_1-i\xi_2+\xi_1\xi_2 & -3i\xi_2+\xi_2^2 & -i\xi_2-i\xi_3+\xi_2\xi_3 \\
 -i\xi_1-i\xi_3+\xi_1\xi_3 & -i\xi_2-i\xi_3+\xi_2\xi_3 & -3i\xi_3+\xi_3^2
\end{array}\right)
\]
where $\xi=(\xi_1,\xi_2,\xi_3)$. The partial derivative of $G$ with respect to $\rho$ at $(\operatorname{id},1)$ in direction of $\bar \rho \in H^{s-1}(\R^3)$ is
\[
 \left. \partial_\rho \right|_{(\operatorname{id},1)} G(\bar \rho) = (1-\Delta)^{-1} \nabla \bar \rho
\]
Hence \eqref{ode_var} reads as
\[
 \frac{d}{dt} \left( \begin{array}{c} \partial \varphi \\ \partial v \end{array} \right) = \left(\begin{array}{cc} 0 & \operatorname{Id}  \\ A & 0 \end{array}\right) \left( \begin{array}{c} \partial \varphi \\ \partial v \end{array} \right) +  \left( \begin{array}{c} 0 \\ (1-\Delta)^{-1} \nabla \bar \rho \end{array} \right) ,\quad \partial \varphi(0) = 0, \partial v(0)=\bar u
\]
By Duhamel's principle the solution at time $T > 0$ is given by
\[
\left( \begin{array}{c} \partial \varphi(T) \\ \partial v(T) \end{array} \right) = e^M \left( \begin{array}{c} 0 \\ \bar u \end{array} \right) + \int_0^T e^{(T-s)M} \left( \begin{array}{c} 0 \\ (1-\Delta)^{-1} \nabla \bar \rho \end{array} \right) \;ds
\]
where
\[
 M=\left(\begin{array}{cc} 0 & \operatorname{Id} \\ A & 0 \end{array}\right)
\]
For $\partial \varphi(T)$ the only contribution comes from the right up entry. We have
\[
 e^M=\left(\begin{array}{cc} \ast & \sum_{k=0}^\infty \frac{T^{2k+1}}{(2k+1)!} A^k \\ \ast & \ast \end{array}\right) \mbox{ resp. }
  e^{(T-s)M}=\left(\begin{array}{cc} \ast & \sum_{k=0}^\infty \frac{(T-s)^{2k+1}}{(2k+1)!} A^k \\ \ast & \ast \end{array}\right)
\]
Integrating gives
\[
 \partial \varphi(T)=\sum_{k=0}^\infty \frac{T^{2k+1}}{(2k+1)!} A^k \bar u + \frac{T^{2k+2}}{(2k+2)!} A^k (1-\Delta)^{-1} \nabla \bar \rho
\]
Consider the operator
\[
 K=\sum_{k=0}^\infty \frac{T^{2k+1}}{(2k+1)!} A^k
\]
It is a Fourier multiplier operator with multiplier
\[
 m(\xi)=\sum_{k=0}^\infty \frac{T^{2k+1}}{(2k+1)!} m_A(\xi)^k
\]
which is a continuous and bounded function with $m(0)=T \cdot \operatorname{Id}$. Therefore $m$ is different from zero. Thus there is $\bar u \in H^s(\R^3;\R^3)$ with 
\[
 K \bar u \neq 0
\]
Now fix $(\rho_\bullet,u_\bullet) \in U_T$ with $u_\bullet$ compactly supported. Take $x^\ast \in \R^3$ with $\operatorname{dist}(x^\ast,\operatorname{supp}(u_\bullet)) > 2$. As $K$ is a Fourier multiplier operator it is translation invariant. So we can find $h_u=\bar u(\cdot+ \Delta x)$ with
\[
 (K h_u)(x^\ast) \neq 0
\]
With this choice of $h_u$ and $h_\rho=0$ we have
\[
 \left(d_{(1,0)}\Psi_T(h_\rho,h_u)\right)(x^\ast)=(\partial \varphi(T))(x^\ast)=(K h_u)(x^\ast) \neq 0
\] 
Take an analytic curve $\gamma=(\gamma_1,\gamma_2):[0,1] \to U_T$ connecting $(1,0)$ with $(\rho_\bullet,u_\bullet)$ with the property that the second component $\gamma_2(t)$ is always compactly supported. This is clearly possible as we can take an arbitrary analytic curve and then multiply $\gamma_2$ with a cut-off function. Now consider the analytic curve
\[
 \alpha:[0,1] \to \R^3,\quad t \mapsto \left(d_{\gamma(t)}\Psi_T(h_\rho,h_u)\right)(x^\ast)
\]
As $\alpha(0) \neq 0$ there is a sequence $t_n \uparrow 1$ with $\alpha(t_n) \neq 0$ for $n \geq 1$. We can put all these $\gamma(t_n)$ to $S$ after moving $x^\ast$ outside of the support of the cut-off function if necessary. By this construction we see that $S \subseteq U_T$ is dense.
\end{proof}

We can now prove Theorem \ref{thm_nonuniform}.
\begin{proof}[Proof of Theorem \ref{thm_nonuniform}]
We fix $(\rho_0,u_0) \in S$ where $S \subseteq U_T$ is as in Lemma \ref{lemma_dense}. In successive steps below we will choose $R_\ast > 0$ and then show that for all $0 < R \leq R_\ast$
\[
 \left. \Phi_T \right|_{B_R((\rho_0,u_0))}
\]
is not uniformly continuous, i.e. that the time $T$ solution map $\Phi_T$ restricted to $B_R((\rho_0,u_0))$ is not uniformly continuous. As $S$ is dense in $U_T$ this suffices clearly to finish the proof.\\
So denote by $\varphi_0=\Psi_T(\rho_0,u_0)$. By continuity we can choose $R_1 > 0$ and $C_1 > 0$ such that
\begin{equation}\label{below_above}
 \frac{1}{C_1} ||\operatorname{curl} u||_{s-1} \leq ||\frac{1}{\det(d\varphi)} [d\varphi] (\operatorname{curl} u) \circ \varphi^{-1}||_{s-1} \leq C_1 ||\operatorname{curl} u||_{s-1}
\end{equation}
for all $u \in B_{R_1}(u_0) \subseteq H^s(\R^3;\R^3)$ and $\varphi \in B_{R_1}(\varphi_0) \subseteq \Ds^s(\R^3)$ where $B_r$ denotes the ball of radius $r$ in the corresponding spaces -- see e.g \cite{composition,lagrangian}. Consider the Taylor expansion
\[
 \Psi_T(\rho_\bullet+h_\rho,u_\bullet+h_u)=\Psi_T(\rho_\bullet,u_\bullet) + d_{(\rho_\bullet,u_\bullet)} \Psi_T (h) + \int_0^1 (1-t) d^2_{(\rho_\bullet+th_\rho,u_\bullet+th_u)} \Psi_T (h,h)
\]
where $h=(h_\rho,h_u) \in H^{s-1}(\R^3) \times H^s(\R^3;\R^3)$. In the following we use the norm
\[
 |||h|||=||h_\rho||_{s-1} + ||h_u||_s
\]
By the smoothness of $\Psi_T$ we can choose $0 < R_2 \leq R_1$ and $C_2 > 0$ with
\[
 ||d^2_{(\rho_\bullet,u_\bullet)}\Psi_T(h_1,h_2)||_s \leq C_2 |||h_1||| \cdot |||h_2|||
\]
and
\[
 ||d^2_{(\widetilde \rho_\bullet,\widetilde u_\bullet)}\Psi_T(h_1,h_2)-d^2_{(\rho_\bullet,u_\bullet)}\Psi_T(h_1,h_2)||_s \leq C_2 |||(\widetilde \rho_\bullet-\rho_\bullet,\widetilde u_\bullet-u_\bullet)||| \cdot |||h_1||| \cdot |||h_2|||
\]
for all $(\rho_\bullet,u_\bullet),(\widetilde \rho_\bullet,\widetilde u_\bullet) \in B_{R_2}((\rho_0,u_0))$ and $h_1, h_2 \in H^{s-1}(\R^3) \times H^s(\R^3;\R^3)$. As $(\rho_0,u_0) \in S$ there is by Lemma \ref{lemma_dense} a corresponding $x^\ast \in \R^3$ with $\operatorname{dist}(\operatorname{supp}u_0,x^\ast) > 2$ and $h=(h_\rho,h_u) \in H^{s-1}(\R^3) \times H^s(\R^3;\R^3)$ with 
\[
 m:=|\left(d_{(\rho_0,u_0)} \Psi_T(h)\right)(x^\ast)| > 0
\]
which we fix. Here $|\cdot|$ denotes the euclidean norm in $\R^3$. Consider the distance 
\[
 d:=\operatorname{dist}(\varphi_0(\operatorname{supp}u_0),\varphi_0(B_1(x^\ast))) > 0
\]
Because of $s > 5/2$ we have by the Sobolev imbedding theorem 
\[
 ||f||_{C^1} \leq \tilde C ||f||_s
\]
for some $\tilde C > 0$. Thus there is $0 < R_3 \leq R_2$ with
\begin{equation}\label{lipschitz}
 |\Psi_T(\rho_\bullet,u_\bullet)(p)-\Psi_T(\rho_\bullet,u_\bullet)(q)| < L |p-q| \quad \forall\; p,q \in \R^3
\end{equation}
and
\begin{equation}\label{dist}
 |\Psi_T(\rho_\bullet,u_\bullet)(p)-\varphi_0(p)| < d/4 \quad \forall\;p \in \R^3
\end{equation}
for all $(\rho_\bullet,u_\bullet) \in B_{R_\ast}((\rho_0,u_0)) \subseteq U_T$. Finally we take $0 < R_\ast \leq R_3$ small enough and $N$ large enough to ensure
\[
 \tilde C C_2 |||h||| \cdot R_\ast^2/4 + \tilde C C_2 \frac{1}{n} |||h||| \cdot R_\ast + \tilde C C_2 \frac{1}{n^2} |||h|||^2 < \frac{m}{2n}
\]
for all $n \geq N$.\\
Now fix $0 < R \leq R_\ast$. We will construct two sequences of initial data $((\rho_0^{(n)},u_0^{(n)}))_{n \geq 1}, ((\tilde \rho_0^{(n)},\tilde u_0^{(n)}))_{n \geq 1} \subseteq B_{R}((u_0,\rho_0))$ with
\[
 |||(\rho_0^{(n)},u_0^{(n)})-(\tilde \rho_0^{(n)},\tilde u_0^{(n)})||| \to 0 \mbox{ as } n \to \infty
\]
whereas
\[
 \limsup_{n \to \infty} |||\Phi_T((\rho_0^{(n)},u_0^{(n)}))-\Phi_T((\tilde \rho_0^{(n)},\tilde u_0^{(n)}))||| > 0
\]
Define the radii $r_n=\dfrac{m}{8nL}$ with the Lipschitz constant from \eqref{lipschitz}. With that take a sequence $(w_n)_{n \geq 1} \subseteq H^s(\R^3;\R^3)$ with
\[
 \operatorname{supp} w_n \subseteq B_{r_n}(x^\ast) \subseteq \R^3 \mbox{ and } ||w_n||_s=R/2
\]
For some technical reason we assume additionally $\operatorname{div} w_n=0$ which is not difficult to arrange. Finally define
\[
 \left(\begin{array}{c} \rho_0^{(n)} \\ u_0^{(n)} \end{array}\right) = \left(\begin{array}{c} \rho_0 \\ u_0 \end{array}\right) + \left(\begin{array}{c} 0 \\ w_n \end{array}\right)
\]
resp.
\[
 \left(\begin{array}{c} \tilde \rho_0^{(n)} \\ \tilde u_0^{(n)} \end{array}\right) = \left(\begin{array}{c} \rho_0 \\ u_0 \end{array}\right) + \left(\begin{array}{c} 0 \\ w_n \end{array}\right) + \frac{1}{n} \left( \begin{array}{c} h_\rho \\ h_u \end{array} \right)
\]
We clearly have 
\[
 (\tilde \rho_0^{(n)},\tilde u_0^{(n)}), (\rho_0^{(n)},u_0^{(n)})\in B_R((\rho_0,u_0)) \quad \forall n \geq N
\]
where $N$ is some large number. Taking $N$ large enough we can assume $r_n \leq 1$ for $n \geq N$. Furthermore by construction
\[
 |||(\rho_0^{(n)},u_0^{(n)})-(\tilde \rho_0^{(n)},\tilde u_0^{(n)})||| \to 0 \mbox{ as } n \to \infty
\]
Let 
\[
 (\rho^{(n)},u^{(n)})=\Phi_T(\rho_0^{(n)},u_0^{(n)}) \quad \mbox{resp.} \quad 
 (\tilde \rho^{(n)},\tilde u^{(n)})=\Phi_T(\tilde \rho_0^{(n)},\tilde u_0^{(n)})  
\]
Similarly
\[
 \varphi^{(n)}=\Psi_T(\rho_0^{(n)},u_0^{(n)}) \quad \mbox{resp.} \quad 
 \tilde \varphi^{(n)}=\Psi_T(\tilde \rho_0^{(n)},\tilde u_0^{(n)})  
\]
We also introduce
\[
 \omega^{(n)} = \operatorname{curl} u^{(n)} \quad \mbox{and} \quad \tilde \omega^{(n)} = \operatorname{curl} \tilde u^{(n)}
\]
As
\[
 |||\Phi_T((\rho_0^{(n)},u_0^{(n)}))-\Phi_T((\tilde \rho_0^{(n)},\tilde u_0^{(n)}))||| \geq ||u^{(n)}-\tilde u^{(n)}||_s \geq C ||\omega^{(n)}-\tilde \omega^{(n)}||_{s-1}
\]
we get the claim by showing
\[
\limsup_{n \to \infty} ||\omega^{(n)}-\tilde \omega^{(n)}||_{s-1} > 0
\]
By \eqref{transport} we have
\[
 \omega^{(n)}=\left(\frac{1}{\det(d\varphi^{(n)})} [d\varphi^{(n)}] \omega_0^{(n)}\right) \circ (\varphi^{(n)})^{-1}
\]
and
\[
 \tilde \omega^{(n)}=\left(\frac{1}{\det(d\tilde \varphi^{(n)})} [d\tilde \varphi^{(n)}] \tilde \omega_0^{(n)}\right) \circ (\tilde \varphi^{(n)})^{-1}
\]
where
\[
 \omega_0^{(n)}=\operatorname{curl} u_0^{(n)} = \operatorname{curl} u_0 + \operatorname{curl} w_n 
\]
and
\[
 \tilde \omega_0^{(n)}=\operatorname{curl} \tilde u_0^{(n)} = \operatorname{curl} u_0 + \operatorname{curl} w_n + \frac{1}{n} \operatorname{curl} h_u
\]
By \eqref{below_above} we have
\[
 ||\left(\frac{1}{\det(d\tilde \varphi^{(n)})} [d\tilde \varphi^{(n)}] \frac{1}{n} \operatorname{curl} h_u \right) \circ (\tilde \varphi^{(n)})^{-1}||_{s-1} \to 0
\]
as $n \to \infty$. Therefore
\begin{align*}
 &\limsup_{n \to \infty} ||\omega^{(n)} - \tilde \omega^{(n)}||_{s-1} =
 \limsup_{n \to \infty} ||\left(\frac{1}{\det(d\varphi^{(n)})} [d\varphi^{(n)}] (\operatorname{curl} u_0 + \operatorname{curl}w_n)\right) \circ (\varphi^{(n)})^{-1} \\
&- \left(\frac{1}{\det(d\tilde \varphi^{(n)})} [d\tilde \varphi^{(n)}] (\operatorname{curl} u_0 + \operatorname{curl}w_n)\right) \circ (\tilde \varphi^{(n)})^{-1}||_{s-1}
\end{align*}
Consider the supports of the above expressions. We have by \eqref{dist}
\[
 \operatorname{supp}\left( (\operatorname{curl}) u_0 \circ (\varphi^{(n)})^{-1}\right),\operatorname{supp}\left( (\operatorname{curl} u_0) \circ (\tilde \varphi^{(n)})^{-1}\right) \subseteq \varphi_0(\operatorname{supp} u_0) + B_{d/4}(0)
\]
where we use $A+B=\{a+b \;|\; a \in A, b \in B\}$. As $\operatorname{supp} w_n \subseteq B_1(x^\ast)$ for $n \geq N$ we have again by \eqref{dist}
\[
 \operatorname{supp}\left( (\operatorname{curl} w_n) \circ (\varphi^{(n)})^{-1}\right),\operatorname{supp}\left( (\operatorname{curl} w_n) \circ (\tilde \varphi^{(n)})^{-1}\right) \subseteq \varphi_0(B_1(x^\ast)) + B_{d/4}(0)
\]
By the the choice of $d$ the supports are in fixed sets which are positive apart. Thus we can ''separate'' the $||\cdot||_{s-1}$ norms with a constant -- see \cite{sqg}. So we have with a constant $C > 0$
\begin{align*}
& \limsup_{n \to \infty} ||\omega^{(n)}-\tilde \omega^{(n)}||_{s-1} \geq 
 C \limsup_{n \to \infty} ||\left(\frac{1}{\det(d\varphi^{(n)})} [d\varphi^{(n)}] (\operatorname{curl}w_n)\right) \circ (\varphi^{(n)})^{-1} \\
& - \left(\frac{1}{\det(d\tilde \varphi^{(n)})} [d\tilde \varphi^{(n)}] (\operatorname{curl}w_n)\right) \circ (\tilde \varphi^{(n)})^{-1}||_{s-1}
\end{align*}
We claim that the supports of the above expressions are also apart. To show this we use the Taylor expansion 
\begin{align*}
 &\tilde \varphi^{(n)} = \Psi(\rho_0,u_0) + d_{(\rho_0,u_0)}\Psi (\left(\begin{array}{c} 0 \\ w_n \end{array}\right) + \frac{1}{n} h ) \\
&+ \int_0^1 (1-t) d^2_{(\rho_0 + t \frac{1}{n} h_\rho,u_0 + t w_n + t \frac{1}{n} h_u)} \Psi (  \left(\begin{array}{c} 0 \\ w_n \end{array}\right) + \frac{1}{n} h,\left(\begin{array}{c} 0 \\ w_n \end{array}\right) + \frac{1}{n} h) \;dt
\end{align*}
and
\begin{align*}
 \varphi^{(n)} = \Psi(\rho_0,u_0) + d_{(\rho_0,u_0)}\Psi (\left(\begin{array}{c} 0 \\ w_n \end{array}\right) ) 
+ \int_0^1 (1-t) d^2_{(\rho_0,u_0 + t w_n)} \Psi (  \left(\begin{array}{c} 0 \\ w_n \end{array}\right),\left(\begin{array}{c} 0 \\ w_n \end{array}\right) ) \;dt
\end{align*}
The difference is
\[
 \tilde \varphi^{(n)} -  \varphi^{(n)} =  \frac{1}{n} d_{(\rho_0,u_0)}\Psi (h) + I_1 + I_2 + I_3
\]
where
\[
 I_1 = \int_0^1 (1-t) \left(d^2_{(\rho_0 + t \frac{1}{n} h_\rho,u_0 + t w_n + t \frac{1}{n} h_u)} \Psi - d_{(\rho_0,u_0+t w_n)} \Psi\right) (  \left(\begin{array}{c} 0 \\ w_n \end{array}\right),\left(\begin{array}{c} 0 \\ w_n \end{array}\right)) \;dt
\]
and
\[
 I_2 = 2 \int_0^1 (1-t) d^2_{(\rho_0 + t \frac{1}{n} h_\rho,u_0 + t w_n + t \frac{1}{n} h_u)} \Psi (  \left(\begin{array}{c} 0 \\ w_n \end{array}\right), \frac{1}{n} h) \;dt
\]
and
\[
 I_3 = \int_0^1 (1-t) d^2_{(\rho_0 + t \frac{1}{n} h_\rho,u_0 + t w_n + t \frac{1}{n} h_u)} \Psi ( \frac{1}{n} h,\frac{1}{n} h) \;dt
\]
Using the estimates for $d^2 \Psi$ from above we have
\[
 ||I_1||_s \leq C_2 \frac{1}{n} |||h||| \cdot ||w_n||_s^2 = C_2 |||h||| \cdot R^2/4
\]
and
\[
 ||I_2||_s  \leq 2 C_2 \frac{1}{n} |||h||| \cdot ||w_n||_s = C_2 \frac{1}{n} |||h||| \cdot R
\]
and
\[
 ||I_3||_s \leq C_2 \frac{1}{n^2} |||h|||^2
\]
By the Sobolev imbedding we then have 
\[
 |I_1(x^\ast)| + |I_2(x^\ast)| + |I_3(x^\ast)| \leq \tilde C C_2 |||h||| \cdot R^2/4 + \tilde C C_2 \frac{1}{n} |||h||| \cdot R + \tilde C C_2 \frac{1}{n^2} |||h|||^2 <\frac{m}{2n}
\]
for $n \geq N$ by the choice of $R_\ast$. Thus we have
\[
  |\tilde \varphi^{(n)}(x^\ast) -  \varphi^{(n)}(x^\ast)| \geq  \frac{1}{n} |\left(d_{(\rho_0,u_0)} \Psi_T(h)\right)(x^\ast)| - \frac{m}{2n} = \frac{m}{2n}
\]
By \eqref{lipschitz} we have
\[
 \operatorname{supp}\left( (\operatorname{curl}w_n) \circ (\varphi^{(n)})^{-1}\right) \subseteq
 B_{Lr_n}(\varphi^{(n)}(x^\ast))=B_{\frac{m}{8n}}(\varphi^{(n)}(x^\ast))
\]
and
\[
 \operatorname{supp}\left( (\operatorname{curl}w_n) \circ (\tilde \varphi^{(n)})^{-1}\right) \subseteq
 B_{Lr_n}(\tilde \varphi^{(n)}(x^\ast))=B_{\frac{m}{8n}}(\tilde \varphi^{(n)}(x^\ast))
\]
So the supports are in a way apart that we can seperate the $H^{s-1}$ norms (see \cite{sqg}) to conclude with a constant $\bar C > 0$
\begin{align*}
  &\limsup_{n \to \infty} ||\left(\frac{[d\varphi^{(n)}]}{\det(d\varphi^{(n)})}  (\operatorname{curl}w_n)\right) \circ (\varphi^{(n)})^{-1} 
 - \left(\frac{[d\tilde \varphi^{(n)}]}{\det(d\tilde \varphi^{(n)})}  (\operatorname{curl}w_n)\right) \circ (\tilde \varphi^{(n)})^{-1}||_{s-1}\\
&\geq \bar C \limsup_{n \to \infty} ||\left(\frac{[d\varphi^{(n)}]}{\det(d\varphi^{(n)})}  (\operatorname{curl}w_n)\right) \circ (\varphi^{(n)})^{-1}||_{s-1}
 + ||\left(\frac{[d\tilde \varphi^{(n)}]}{\det(d\tilde \varphi^{(n)})}  (\operatorname{curl}w_n)\right) \circ (\tilde \varphi^{(n)})^{-1}||_{s-1}
\end{align*}
Using \eqref{below_above} we can estimate this from below by
\[
 \frac{\bar C}{C_2} \limsup_{n \to \infty} ||\operatorname{curl}w_n||_{s-1} \geq \frac{\bar C \hat C}{C_2} \limsup_{n \to \infty} ||dw_n||_{s-1} 
\]
where the last inequality with some $\hat C > 0$ follows from the Biot-Savart law (see \cite{chemin}) for divergence free vector fields (Here is where we use $\operatorname{div} w_n=0$). We have the following general equivalence 
\[
 ||w_n||_s \sim ||w_n||_{L^2} + ||dw_n||_{s-1}
\]
But 
\[
 ||w_n||_{L^2} \leq ||w_n||_{L^\infty}  \sqrt{\frac{4}{3} \pi r_n^3} \leq \tilde C ||w_n||_s \sqrt{\frac{4}{3} \pi r_n^3} \to 0
\]
since $w_n$ is supported in $B_{r_n}(x^\ast)$. Therefore
\[
 \limsup_{n \to \infty} ||dw_n||_{s-1} \geq K \limsup_{n \to \infty} ||w_n||_s \geq K R/2 
\]
for some $K > 0$. Altogether we have
\[
 \limsup_{n \to \infty} ||\tilde \omega^{(n)}- \omega^{(n)}||_{s-1} \geq \frac{\bar C \hat C K}{2 C_2} R
\]
showing the claim.
\end{proof}

\bibliographystyle{plain}

\flushleft
\author{ Hasan \.{I}nci\\
Ko\c{c} \"Universitesi Matematik b\"ol\"um\"u\\
Rumelifeneri Yolu\\
34450 Sar{\i}yer \.{I}stanbul T\"urkiye\\
        {\it email: } {hinci@ku.edu.tr}
}

\end{document}